\documentclass[11pt]{amsart}

\usepackage{graphicx,amscd,amsmath,amsfonts,amssymb,geometry,color}
\usepackage[initials]{amsrefs}

\newtheorem{theorem}{Theorem}[section]

\newtheorem{corollary}[theorem]{Corollary}

\numberwithin{equation}{section}

\begin{document}

\title[Universal bounds for the Hardy--Littlewood inequalities on multilinear forms]{Universal bounds for the Hardy--Littlewood inequalities on multilinear forms}

\author[Ara\'{u}jo]{G. Ara\'{u}jo}
\address[G. Ara\'{u}jo]{Universidade Estadual da Para\'{\i}ba, Campina Grande - PB (58.429-500), Brazil}
\email{gdasaraujo@gmail.com}

\author[C\^amara]{K. C\^amara}
\address[K. C\^amara]{Universidade Federal da Para\'{\i}ba and Universidade Federal Rural do Semi-\'Arido, Jo\~{a}o Pessoa - PB (58.051-900) and Mossor\'o - RN (59.625-900), Brazil}
\email{kleber.soares@ufersa.edu.br}



\begin{abstract}
The Hardy--Littlewood inequalities for multilinear forms on sequence spaces state that for all positive integers $m,n\geq2$ and all $m$-linear forms $T:\ell_{p_{1}}^{n}\times\cdots\times\ell_{p_{m}}^{n}\rightarrow\mathbb{K}$ ($\mathbb{K}=\mathbb{R}$ or $\mathbb{C}$) there are constants $C_{m}\geq1$ (not depending on $n$) such that
\[
\left(  \sum_{j_{1},\ldots,j_{m}=1}^{n}\left\vert T(e_{j_{1}},\ldots,e_{j_{m}})\right\vert ^{\rho}\right)  ^{\frac{1}{\rho}}\leq C_{m}\sup_{\left\Vert x_{1}\right\Vert ,\dots,\left\Vert x_{m}\right\Vert \leq 1}\left\vert T(x_{1},\dots,x_{m})\right\vert,
\]
where $\rho=\frac{2m}{m+1-2\left(  \frac{1}{p_{1}}+\cdots+\frac{1}{p_{m}}\right)  }$ if $0\leq\frac{1}{p_{1}}+\cdots+\frac{1}{p_{m}}\leq\frac{1}{2}$ or $\rho=\frac{1}{1-\left(  \frac{1}{p_{1}}+\cdots+\frac{1}{p_{m}}\right)}$ if $\frac{1}{2}\leq\frac{1}{p_{1}}+\cdots+\frac{1}{p_{m}}<1$. Good estimates for the Hardy-Littlewood constants are, in general, associated to applications in Mathematics and even in Physics, but the exact behavior of these constants is still unknown. In this note we give some new contributions to the behavior of the constants in the case $\frac{1}{2}\leq\frac{1}{p_{1}}+\cdots+\frac{1}{p_{m}}<1$. As a consequence of our main result, we present a generalization and a simplified proof of a result due to Aron et al. on certain Hardy--Littlewood type inequalities.
\end{abstract}

\maketitle

\section{Introduction}

Let $E,E_{1},...,E_{m}$ and $F$ be Banach spaces over $\mathbb{K}=\mathbb{R}$
or $\mathbb{C}$ and for all $m$-linear maps $T:E_{1}\times\cdots\times E_{m}
\to F$ let us denote
\[
\left\|  T\right\|  := \sup_{ \left\Vert x_{1}\right\Vert , \dots, \left\Vert
	x_{m}\right\Vert \leq1} \left\Vert T(x_{1}, \dots, x_{m})\right\Vert .
\]
Also, let $c_{0}=\left\{  \left(  x_{n}\right)  _{n=1}^{\infty}\subset
\mathbb{K}:\lim x_{n}=0\right\}  $. \textit{Littlewood's $4/3$ inequality}
\cite{LLL} (1930) asserts that
\[
\left(  \sum_{j,k=1}^{\infty}\left|  T(e_{j},e_{k})\right|  ^{\frac{4}{3}%
}\right)  ^{\frac{3}{4}}\leq\sqrt{2}\|T\|,
\]
for all continuous bilinear forms $T \colon c_{0}\times c_{0} \to\mathbb{C}$,
and the exponent $4/3$ is sharp.

Littlewood's $4/3$ inequality was the starting point of several important
inequalities, such as an inequality due to Bohnenblust and Hille (1931), which
nowadays is known to be important for applications in physics (see
\cite{Montanaro}). The \textit{Bohnenblust--Hille inequality} \cite{bh}
assures the existence of a constant $B_{m}\geq1$ such that
\[
\left(  \sum_{j_{1},\ldots,j_{m}=1}^{\infty}\left\vert T(e_{j_{1}}%
,\ldots,e_{j_{m}})\right\vert ^{\frac{2m}{m+1}}\right)  ^{\frac{m+1}{2m}}\leq
B_{m}\left\Vert T\right\Vert ,
\]
for all continuous $m$--linear forms $T\colon c_{0}\times\cdots\times c_{0}
\to\mathbb{C}$.

Of course, if $m=2$ we recover Littlewood's $4/3$ inequality. In 1934 Hardy
and Littlewood \cite{hl} extended Littlewood's $4/3$ inequality to bilinear
maps defined on $\ell_{p}\times\ell_{q}$, where by $\ell_{s}$, $s\geq1$, we
mean the Banach space of all absolutely $s$--summable sequences in
$\mathbb{K}$ (of course, if $s=\infty$ by $\ell_{\infty}$ we mean the space of
all bounded sequences in $\mathbb{K}$). In 1981, Praciano-Pereira \cite{pra}
extended the Hardy--Littlewood inequalities to $m$-linear forms on
$\ell_{p_{1}}\times\cdots\times\ell_{p_{m}}$ for $0\leq\frac{1}{p_{1}}%
+\cdots+\frac{1}{p_{m}}\leq\frac{1}{2}$ and very recently Dimant and
Sevilla-Peris \cite{dimant} generalized the estimates for the case $\frac
{1}{2}\leq\frac{1}{p_{1}}+\cdots+\frac{1}{p_{m}}<1$ (all these inequalities
are nowadays called \textit{Hardy--Littlewood inequalities}).

From now on, for any function $f$, whenever it makes sense we formally define
$f(\infty)=\lim_{p\rightarrow\infty}f(p)$. Moreover, for $\mathbf{p}%
=(p_{1},\dots,p_{m})\in[1,\infty]^{m}$ and $1\leq k\leq m$, let us denote
\[
\left|  \frac{1}{\mathbf{p}}\right|  _{\leq k}:=\frac{1}{p_{1}}+\dots+\frac
{1}{p_{k}}, \quad\left|  \frac{1}{\mathbf{p}}\right|  _{\geq k}:=\frac
{1}{p_{k}}+\dots+\frac{1}{p_{m}} \quad\text{and} \quad\left|  \frac
{1}{\mathbf{p}} \right|  := \left|  \frac{1}{\mathbf{p}} \right|  _{\leq
	m}=\left|  \frac{1}{\mathbf{p}} \right|  _{\geq1}
\]
and, as usual, for $s\in[1,\infty]$ and a positive integer $n$ we define
$\ell_{s}^{n}=\mathbb{K}^{n}$ equipped with the $\ell_s$-norm ($\sup$ norm if
$s=\infty$); also, $e_{j}$ represents the canonical vector of $c_{0}$ with $1$
in the $j$-th coordinate and $0$ elsewhere.

The classical Hardy--Littlewood inequalities can be stated as follows:

\begin{theorem}
	[Bohnenblust, Dimant, Hardy, Hille, Littlewood, Praciano-Pereira,
	Sevilla-Perez]Let $m\geq2$ be a positive integer and $\mathbf{p}=(p_{1}%
	,\ldots,p_{m})\in(1,\infty]^{m}$ with $0\leq\left|  \frac{1}{\mathbf{p}%
	}\right|  <1$. Then there are constants $C_{m,\mathbf{p}}^{\mathbb{K}}\geq1$
	such that
	\begin{align}
	&  \left(  \sum_{j_{1},\ldots,j_{m}=1}^{n}\left\vert T(e_{j_{1}}%
	,\ldots,e_{j_{m}})\right\vert ^{\frac{2m}{m+1-2\left|  \frac{1}{\mathbf{p}%
			}\right|  }}\right)  ^{\frac{m+1-2\left|  \frac{1}{\mathbf{p}}\right|  }{2m}%
	}\leq C_{m,\mathbf{p}}^{\mathbb{K}}\left\Vert T\right\Vert \text{ if }
	0\leq\left|  \frac{1}{\mathbf{p}}\right|  \leq\frac{1}{2},\label{hl1}\\
	&  \left(  \sum_{j_{1},\ldots,j_{m}=1}^{n}\left\vert T\left(  e_{j_{1}}%
	,\ldots,e_{j_{m}}\right)  \right\vert ^{\frac{1}{1-\left|  \frac{1}%
			{\mathbf{p}}\right|  }}\right)  ^{1-\left|  \frac{1}{\mathbf{p}}\right|  }\leq
	C_{m,\mathbf{p}}^{\mathbb{K}}\left\Vert T\right\Vert \text{ if }\frac{1}%
	{2}\leq\left|  \frac{1}{\mathbf{p}}\right|  <1, \label{hl2}%
	\end{align}
	for all $m$-linear forms $T:\ell_{p_{1}}^{n}\times\cdots\times\ell_{p_{m}}%
	^{n}\rightarrow\mathbb{K}$ and all positive integers $n$.
	
\end{theorem}

If $p_1=\cdots=p_m=p$ we denote $C_{m,\mathbf{p}}^\mathbb{K}$ by $C_{m,p}^\mathbb{K}$. When $\left|\frac{1}{\mathbf{p}}\right|=0$ (equivalently $p_1=\cdots=p_m=\infty$), since $\frac{2m}{m+1-2\left|\frac{1}{\mathbf{p}}\right|}=\frac{2m}{m+1}$, we recover the classical Bohnenblust--Hille inequality. Using the generalized Kahane--Salem--Zygmund inequality in \eqref{hl1} and H\"older's inequality in \eqref{hl2} it is possible to conclude that the exponents $\frac{2m}{m+1-2\left|\frac{1}{\mathbf{p}}\right|}$ and $\frac{1}{1-\left|\frac{1}{\mathbf{p}}\right|}$ are optimal: if replaced by smaller exponents the constants appearing on the right-hand-size will depend on $n$.


The precise growth of the constants $C_{m,\mathbf{p}}^{\mathbb{K}}$, $0\leq \left|\frac{1}{\mathbf{p}}\right|<1$, is important for many applications and remains an open problem in Mathematical Analysis. The first estimates for $C_{m,\mathbf{p}}^{\mathbb{K}}$ had exponential growth; more precisely,
\[
C_{m,\mathbf{p}}^{\mathbb{K}}\leq \left( \sqrt{2}\right) ^{m-1}.
\]
The case $0\leq\left|\frac{1}{\mathbf{p}}\right|\leq \frac{1}{2}$ was more explored since it appearance. Several studies have made significant progress in the context $0\leq\left|\frac{1}{\mathbf{p}}\right|\leq \frac{1}{2}$ (see for instance \cite{alb,n,ap,bayart,ap2}). For example, among other results, it was proved in \cite{bayart,ap2} that for $2m(m-1)^{2}<p\leq \infty$ we have
\begin{align*}
& C_{m,p}^{\mathbb{R}} < \kappa _{1}\cdot m^{\frac{2-\log 2-\gamma }{2}}\approx \kappa _{1}\cdot m^{0.36482}, \\
& C_{m,p}^{\mathbb{C}} < \kappa _{2}\cdot m^{\frac{1-\gamma }{2}}\approx \kappa _{2}\cdot m^{0.21139},
\end{align*}
for certain constants $\kappa _{1},\kappa _{2}>0,$ where $\gamma$ is the Euler-Mascheroni constant.

On the other hand, the case $\frac{1}{2}\leq\left|\frac{1}{\mathbf{p}}\right|<1$ was virtually unexplored and only recently in \cite{UNIV, anps} is that the original estimate was improved. Our main result generalizes some of the main results of \cite{UNIV, anps}.

One of the main results of \cite{UNIV} is the following result:

\begin{theorem}\label{theor3.3}
Let $m\geq 2$ be a positive integer and $\mathbf{p}=(p_1,\ldots,p_m)\in(1,\infty]$ with $\frac{1}{2}\leq \left|\frac{1}{\mathbf{p}}\right|<1$. Then, for all $m$-linear forms $T:\ell _{p_{1}}^{n}\times \cdots \times \ell_{p_{m}}^{n}\rightarrow \mathbb{K}$ and all positive integers $n$,
\begin{equation*}
\left( \sum_{j_{1},\ldots,j_{m}=1}^{n}\left\vert T\left( e_{j_{1}},\ldots,e_{j_{m}}\right) \right\vert ^{\frac{1}{1-\left|\frac{1}{\mathbf{p}}\right|}}\right) ^{1-\left|\frac{1}{\mathbf{p}}\right| }  \leq 2^{(m-1)\left( 1-\left|\frac{1}{\mathbf{p}}\right| \right) }\left\Vert T\right\Vert.
\end{equation*}
\end{theorem}

As a consequence, when $m<p_1=\cdots=p_m=p\leq m+1$, the optimal constants of the Hardy--Littlewood inequalities are uniformly bounded by $2$. In fact, for $m<p\leq m+1$ we have
\[
\left( \sum_{j_{1},\ldots,j_{m}=1}^{n}\left\vert T\left( e_{j_{1}},\ldots,e_{j_{m}}\right) \right\vert ^{\frac{p}{p-m}}\right) ^{\frac{p-m}{p}}\leq 2^{\frac{m-1}{m+1}}\left\Vert T\right\Vert <2\left\Vert T\right\Vert,
\]
for all $m$-linear forms $T:\ell _{p}^{n}\times \cdots \times \ell_{p}^{n}\rightarrow \mathbb{K}$ and all positive integers $n$.

Another important contribution in this setting ($\frac{1}{2}\leq \left|\frac{1}{\mathbf{p}}\right|<1$) is the following result of Aron, N\'u\~nez-Alarc\'on, Pellegrino and Serrano-Rodr\'iguez (see \cite[Corollary 3.3]{anps}):

\begin{theorem}\label{corrrr}
Let $m\geq 2$ be a positive integer and $\mathbf{p}=(p_{1},\ldots ,p_{m})\in (1,\infty]^m$ be such that $1<p_m\leq 2 <p_1,\ldots,p_{m-1}$ and
\[
\frac{1}{2}\leq \left|\frac{1}{\mathbf{p}}\right|<1.
\]
Then
\[
\left( \sum_{j_{1},\ldots,j_{=m}}^{n}|T(e_{j_{1}},\ldots,e_{j_{m}})|^{\frac{1}{1-\left|\frac{1}{\mathbf{p}}\right|}}\right) ^{1-\left|\frac{1}{\mathbf{p}}\right|}\leq \Vert T\Vert,
\]
for all $m$-linear forms $T:\ell _{p_{1}}^{n}\times \cdots \times \ell_{p_{m}}^{n}\rightarrow \mathbb{K}$ and all positive integers $n$.
\end{theorem}

Our main result generalizes Theorem \ref{theor3.3} and has as a consequence a more general result than Theorem \ref{corrrr}. It is important to mention that the proof of our main result is not just an adaptation of the original proof of \ref{theor3.3} and that the proof given in \cite{anps} for Theorem \ref{corrrr} is, in some sense, very extensive and complicated. Our approach is simpler and more self-contained.

\section{Main results}

We begin this section by recalling some important auxiliary results that will be essential to our purpose.

An important auxiliary result that will be used along this note is the Khinchine inequality for real and complex scalars. More precisely, the Khinchine inequality assures that for any $0<q<\infty $, there are positive constants $A_{q}^{\mathbb{K}}$ such that regardless of the positive integer $n$ and of the scalar sequence $(a_{j})_{j=1}^{n}$ we have
\begin{equation*}
A_{q}\left( \sum\limits_{j=1}^{n}|a_{j}|^{2}\right) ^{\frac{1}{2}}\leq \left( \int_{0}^{1}\left\vert \sum\limits_{j=1}^{n}a_{j}r_{j}(t)\right\vert ^{q}dt\right) ^{\frac{1}{q}},
\end{equation*}
where $r_{j}$ are the Rademacher functions.

The next result concerns the multilinear theory of absolutely summing operators initiated by Pietsch \cite{pi1}. It was proved very recently by Albuquerque and Rezende in \cite[Theorem 3]{ALBLIZ} and also will be essential for us. First, let us present some required definitions. Let $B_{E^{\ast }}$ be the closed unit ball of the topological dual of $E$. If $1\leq q\leq \infty $, the symbol $q^{\ast}$ represents the conjugate of $q$. It will be convenient to adopt that $\frac{c}{\infty }=0$ for any $c>0$; for $s\geq 1$ we represent by $\ell _{s}^{w}(E)$ the linear space of the sequences $\left( x_{j}\right)_{j=1}^{\infty }$ in $E$ such that $\left( \varphi \left( x_{j}\right)\right) _{j=1}^{\infty }\in \ell _{s}$ for every continuous linear functional $\varphi :E\rightarrow \mathbb{K}$. For $\left( x_{j}\right) _{j=1}^{\infty }\in \ell _{s}^{w}(E)$ the expression $\|(x_j)_{j=1}^\infty\|_{w,s}:=\sup_{\varphi \in B_{E^{\ast }}}\Vert \left( \varphi \left(x_{j}\right) \right) _{j=1}^{\infty }\Vert _{s}$ defines a norm on $\ell _{s}^{w}(E)$. The space of all continuous $m$-linear operators $T:E_{1}\times \cdots \times E_{m}\rightarrow F$, with the $\sup$ norm, is denoted by $\mathcal{L}\left( E_{1},...,E_{m};F\right)$. For $\mathbf{p},\mathbf{q}\in\lbrack1,+\infty)^{m}$, a multilinear operator $T:E_{1}\times\cdots\times E_{m}\rightarrow F$ is multiple $(\mathbf{q};\mathbf{p})$-summing if there exist a constant $C>0$ such that
\[
\left( \sum\limits_{j_{1}=1}^{\infty}\left( \cdots\left( \sum\limits_{j_{m}=1}^{\infty}\left\Vert T( x_{j_{1}}^{(1)},\dots,x_{j_{m}}^{(m)})\right\Vert_{F}^{q_{m}}\right)^{\frac{q_{m-1}}{q_{m}}}\cdots\right)^{\frac{q_{1}}{q_{2}}}\right)^{\frac{1}{q_{1}}}\leq C \prod\limits_{k=1}^{m}\left\Vert(x_{j}^{(k)})_{j=1}^{\infty}\right\Vert _{w,p_{k}}
\]
for all $(x_{j}^{(k)})_{j=1}^{\infty}\in\ell_{p_{k}}^{w}\left( E_{k}\right)$. We represent the class of all multiple $(\mathbf{q};\mathbf{p})$-summing operators by $\Pi_{(\mathbf{q};\mathbf{p})}^m\left(E_{1},\dots,E_{m};F\right)$. When $q_1=\cdots=q_m=q$, we denote $\Pi_{(\mathbf{q};\mathbf{p})}^m\left(E_{1},\dots,E_{m};F\right)$ by $\Pi_{\left(q;\mathbf{p}\right) }^m\left(
E_{1},\dots,E_{m};F\right)$. For recent results on the theory of multiple $(\mathbf{q};\mathbf{p})$-summing operators we refer to \cite{PELLEGRINO2017802}.

\begin{theorem}[Albuquerque and Rezende \cite{ALBLIZ}]\label{ALBLIZZ}
Let $m$ be a positive integer and $r\geq 1,\mathbf{s},\mathbf{p},\mathbf{q}\in[1,\infty)^m$ be such that
\[
\frac{1}{r}-\left|\frac{1}{\mathbf{p}}\right|+\left|\frac{1}{\mathbf{q}}\right|>0
\]
and, for each $k=1,\ldots,m$, $q_k\geq p_k$ and
\[
\frac{1}{s_k}-\left|\frac{1}{\mathbf{q}}\right|_{\geq k}=\frac{1}{r}-\left|\frac{1}{\mathbf{p}}\right|_{\geq k}.
\]
Then
\[
\Pi_{(r;\mathbf{p})}^m(E_1,\ldots,E_m;F)\subset \Pi_{(\mathbf{s},\mathbf{q})}^m(E_1,\ldots,E_m;F)
\]
for any Banach spaces $E_1,\ldots,E_m,F$ and the inclusion operator has norm $1$.
\end{theorem}

Now we are able to present our main result.

\begin{theorem}
Let $m\geq 2$ be a positive integer and $\mathbf{p}=(p_{1},\ldots ,p_{m})\in (1,\infty]^m$ be such that
\[
\frac{1}{2}\leq \left|\frac{1}{\mathbf{p}}\right|<1.
\]
Then
\[
\left( \sum_{j_{1},\ldots,j_{=m}}^{n}|T(e_{j_{1}},\ldots,e_{j_{m}})|^{\frac{1}{1-\left|\frac{1}{\mathbf{p}}\right|}}\right) ^{1-\left|\frac{1}{\mathbf{p}}\right|}\leq 2^{(s-1)\left[1-\left(\frac{1}{p_{k_1}}+\cdots+\frac{1}{p_{k_s}}\right)\right]}\Vert T\Vert,
\]
for all $m$-linear forms $T:\ell _{p_{1}}^{n}\times \cdots \times \ell_{p_{m}}^{n}\rightarrow \mathbb{K}$ and all positive integers $n$, where
\[
s=\min\left\{r \ : 
\begin{array}{l}
\text{there exists }p_{k_1},\ldots,p_{k_r}\in\{p_1,\ldots,p_m \} \text{ with }p_{k_i}\neq p_{k_j}, i\neq j,\\
\text{ and }\frac{1}{2}\leq \frac{1}{p_{k_1}}+\cdots+\frac{1}{p_{k_r}}<1
\end{array}
\right\}
\]
\end{theorem}

\begin{proof}
For the sake of simplicity let us suppose that $p_{k_1}=p_1,\ldots,p_{k_s}=p_s$. Since $$\frac{1}{2}\leq \left|\frac{1}{\mathbf{p}}\right|_{\leq s}<1,$$ it follows from the Theorem \ref{theor3.3} that
\[
\left( \sum_{j_1,\ldots,j_s=1}^{n}|T_s(e_{j_1},\ldots, e_{j_s})|^{\frac{1}{1-\left|\frac{1}{\mathbf{p}}\right|_{\leq s}}}\right) ^{1-\left|\frac{1}{\mathbf{p}}\right|_{\leq s}}\leq 2^{(s-1)\left[1-\left|\frac{1}{\mathbf{p}}\right|_{\leq s}\right]}\Vert T_s\Vert
\]
for all $s$-linear forms $T_s:\ell _{p_{1}}^{n}\times \cdots \times \ell_{p_{s}}^{n}\rightarrow \mathbb{K}$ and all positive integers $n$. In view of the Kinchine's inequality we have, for every $n$ and all $(s+1)$-linear forms $T_{s+1}:\ell _{p_{1}}^{n}\times \cdots \times\ell_{p_{s}}^{n}\times\ell_{\infty}^{n} \rightarrow \mathbb{K}$,
\begin{align*}
& \left( \sum\limits_{j_1,\ldots,j_s=1}^{n}\left( \sum\limits_{j_{s+1}=1}^{n}\left\vert T_{s+1}\left( e_{j_1},\ldots,e_{j_{s+1}}\right) \right\vert ^{2}\right) ^{\frac{1}{2}\cdot\frac{1}{1-\left|\frac{1}{\mathbf{p}}\right|_{\leq s}}}\right)^{1-\left|\frac{1}{\mathbf{p}}\right|_{\leq s}} \\
&\leq \left(\sum\limits_{j_1,\ldots,j_s=1}^{n} A_{\mathbb K ,\frac{1}{1-\left|\frac{1}{\mathbf{p}}\right|_{\leq s}}}^{-1}\left(\int_{0}^{1}\left\vert\sum\limits_{j_{s+1}=1}^{n}T_{s+1}\left( e_{j_1},\ldots,e_{j_{s+1}}\right)r_{j_{s+1}}(t)\right\vert ^{\frac{1}{1-\left|\frac{1}{\mathbf{p}}\right|_{\leq s}}}dt\right) ^{\frac{1-\left|\frac{1}{\mathbf{p}}\right|_{\leq s}}{1-\left|\frac{1}{\mathbf{p}}\right|_{\leq s}}}\right)^{1-\left|\frac{1}{\mathbf{p}}\right|_{\leq s}} \\
& = A_{\mathbb K ,\frac{1}{1-\left|\frac{1}{\mathbf{p}}\right|_{\leq s}}}^{-1}\left(\int_{0}^{1}\sum\limits_{j_1,\ldots,j_s=1}^{n} \left\vert T_{s+1}\left( e_{j_1},\ldots,e_{j_s},\sum\limits_{j_{s+1}=1}^{n} e_{j_{s+1}}r_{j_{s+1}}(t)\right)\right\vert ^{\frac{1}{1-\left|\frac{1}{\mathbf{p}}\right|_{\leq s}}}dt\right)^{1-\left|\frac{1}{\mathbf{p}}\right|_{\leq s}} \\
& \leq A_{\mathbb K ,\frac{1}{1-\left|\frac{1}{\mathbf{p}}\right|_{\leq s}}}^{-1}\left(\sup_{t\in[0,1]}\sum\limits_{j_1,\ldots,j_s=1}^{n} \left\vert T_{s+1}\left( e_{j_1},\ldots,e_{j_s},\sum\limits_{j_{s+1}=1}^{n} e_{j_{s+1}}r_{j_{s+1}}(t)\right)\right\vert ^{\frac{1}{1-\left|\frac{1}{\mathbf{p}}\right|_{\leq s}}}dt\right)^{1-\left|\frac{1}{\mathbf{p}}\right|_{\leq s}} \\
& = A_{\mathbb K ,\frac{1}{1-\left|\frac{1}{\mathbf{p}}\right|_{\leq s}}}^{-1}\sup_{t\in[0,1]}\left(\sum\limits_{j_1,\ldots,j_s=1}^{n} \left\vert T_{s+1}\left( e_{j_1},\ldots,e_{j_s},\sum\limits_{j_{s+1}=1}^{n} e_{j_{s+1}}r_{j_{s+1}}(t)\right)\right\vert ^{\frac{1}{1-\left|\frac{1}{\mathbf{p}}\right|_{\leq s}}}dt\right)^{1-\left|\frac{1}{\mathbf{p}}\right|_{\leq s}} \\
& \leq A_{\mathbb K ,\frac{1}{1-\left|\frac{1}{\mathbf{p}}\right|_{\leq s}}}^{-1}2^{(s-1)\left[1-\left|\frac{1}{\mathbf{p}}\right|_{\leq s}\right]}\sup_{t\in[0,1]}\left\Vert T_{s+1}\left( \cdot ,\ldots,\cdot,\sum\limits_{j_{s+1}=1}^{n} e_{j_{s+1}}r_{j_{s+1}}(t)\right)\right\Vert \\
& = A_{\mathbb K ,\frac{1}{1-\left|\frac{1}{\mathbf{p}}\right|_{\leq s}}}^{-1}2^{(s-1)\left[1-\left|\frac{1}{\mathbf{p}}\right|_{\leq s}\right]}\|T_{s+1}\|,
\end{align*}
where $A_{\mathbb K ,\frac{1}{1-\left|\frac{1}{\mathbf{p}}\right|_{\leq s}}}$ is the constant of the Khinchine inequality.

Since 
$$
\frac{1}{1-\left|\frac{1}{\mathbf{p}}\right|_{\leq s}}\geq 2,
$$
we have $A_{\mathbb K ,\frac{1}{1-\left|\frac{1}{\mathbf{p}}\right|_{\leq s}}}=1$ and thus (from the previous inequality together with canonical inclusion of $\ell_p$ spaces)
\begin{align*}
&\left( \sum\limits_{j_1,\ldots,j_{s+1}=1}^{n} \left\vert T_{s+1}\left( e_{j_1},\ldots,e_{j_{s+1}}\right) \right\vert ^{\frac{1}{1-\left|\frac{1}{\mathbf{p}}\right|_{\leq s}}}\right)^{1-\left|\frac{1}{\mathbf{p}}\right|_{\leq s}}\\
&=\left( \sum\limits_{j_1,\ldots,j_s=1}^{n}\left( \sum\limits_{j_{s+1}=1}^{n}\left\vert T_{s+1}\left( e_{j_1},\ldots,e_{j_{s+1}}\right) \right\vert ^{\frac{1}{1-\left|\frac{1}{\mathbf{p}}\right|_{\leq s}}}\right) ^{\left(1-\left|\frac{1}{\mathbf{p}}\right|_{\leq s}\right)\cdot\frac{1}{1-\left|\frac{1}{\mathbf{p}}\right|_{\leq s} }}\right)^{1-\left|\frac{1}{\mathbf{p}}\right|_{\leq s}}\\
& \leq \left( \sum\limits_{j_1,\ldots,j_s=1}^{n}\left( \sum\limits_{j_{s+1}=1}^{n}\left\vert T_{s+1}\left( e_{j_1},\ldots,e_{j_{s+1}}\right) \right\vert ^{2}\right) ^{\frac{1}{2}\cdot\frac{1}{1-\left|\frac{1}{\mathbf{p}}\right|_{\leq s}}}\right)^{1-\left|\frac{1}{\mathbf{p}}\right|_{\leq s}} \\
& \leq 2^{(s-1)\left[1-\left|\frac{1}{\mathbf{p}}\right|_{\leq s}\right]}\|T_{s+1}\|,
\end{align*}
for every $n$ and all $(s+1)$-linear forms $T_{s+1}:\ell _{p_{1}}^{n}\times \cdots \times\ell_{p_{s}}^{n}\times\ell_{\infty}^{n} \rightarrow \mathbb{K}$. Using the canonical isometric isomorphisms for the spaces of weakly summable sequences (see \cite[Proposition 2.2]{diestel}) we know that this is equivalent to assert that (see \cite[p. 308]{dimant}),
\[
\Pi_{\left(\frac{1}{1-\left|\frac{1}{\mathbf{p}}\right|_{\leq s}};p_1^* ,\ldots,p_s^*,1\right)}^{s+1}(E_1,\ldots,E_{s+1};\mathbb{K})=\mathcal{L}(E_1,\ldots,E_{s+1};\mathbb{K})
\]
for all Banach spaces $E_1,\ldots,E_{s+1}$.

From Theorem \ref{ALBLIZZ} it is possible to prove that
\[
\Pi_{\left(\frac{1}{1-\left|\frac{1}{\mathbf{p}}\right|_{\leq s}};p_1^* ,\ldots,p_s^*,1\right)}^{s+1}(E_1,\ldots,E_{s+1};\mathbb{K})\subseteq \Pi_{\left(\frac{1}{1-\left|\frac{1}{\mathbf{p}}\right|_{\leq s+1}};p_1^* ,\ldots,p_{s+1}^*\right)}^{s+1}(E_1,\ldots,E_{s+1};\mathbb{K}).
\]
Consequently,
\[
\Pi_{\left(\frac{1}{1-\left|\frac{1}{\mathbf{p}}\right|_{\leq s+1}};p_1^* ,\ldots,p_{s+1}^*\right)}^{s+1}(E_1,\ldots,E_{s+1};\mathbb{K})=\mathcal{L}(E_1,\ldots,E_{s+1};\mathbb{K})
\]
for all Banach spaces $E_1,\ldots,E_{s+1}$. Again (see \cite[p. 308]{dimant}), this is equivalent to say that
\begin{align*}
&\left( \sum\limits_{j_1,\ldots,j_{s+1}=1}^{n} \left\vert T_{s+1}\left( e_{j_1},\ldots,e_{j_{s+1}}\right) \right\vert ^{\frac{1}{1-\left|\frac{1}{\mathbf{p}}\right|_{\leq s+1}}}\right)^{1-\left|\frac{1}{\mathbf{p}}\right|_{\leq s+1}}\\
& \leq 2^{(s-1)\left[1-\left|\frac{1}{\mathbf{p}}\right|_{\leq s}\right]}\|T_{s+1}\|,
\end{align*}
for all $(s+1)$-linear forms $T_{s+1}:\ell _{p_{1}}^{n}\times \cdots \times\ell_{p_{s}}^{n}\times\ell_{\infty}^{n} \rightarrow \mathbb{K}$ and all positive integers $n$.

The proof is completed by a standard induction argument.
\end{proof}

Just making $s=1$ in the previous result, we get the following Hardy--Littlewood type inequalities with constant $1$:

\begin{corollary}\label{corolario}
Let $m\geq 2$ be a positive integer and $\mathbf{p}=(p_{1},\ldots ,p_{m})\in (1,\infty]^m$ be such that $1<p_i\leq 2 <p_1,\ldots,p_{i-1},p_{i+1},\ldots,p_m$ for some $1\leq i\leq m$ and
\[
\frac{1}{2}\leq \left|\frac{1}{\mathbf{p}}\right|<1.
\]
Then
\[
\left( \sum_{j_{1},\ldots,j_{=m}}^{n}|T(e_{j_{1}},\ldots,e_{j_{m}})|^{\frac{1}{1-\left|\frac{1}{\mathbf{p}}\right|}}\right) ^{1-\left|\frac{1}{\mathbf{p}}\right|}\leq \Vert T\Vert,
\]
for all $m$-linear forms $T:\ell _{p_{1}}^{n}\times \cdots \times \ell_{p_{m}}^{n}\rightarrow \mathbb{K}$ and all positive integers $n$.
\end{corollary}

Corollary \ref{corolario} generalizes a recently result proved independently in \cite[Corollary 3.3]{anps}. Our approach is different and we believe it is more self-contained.

\end{document}